\documentclass[english]{ourlematema}

\usepackage{amsmath, amsthm, amssymb, hyperref, color}
\usepackage{amsfonts,mathrsfs}
 \usepackage{amsthm}
 \usepackage{amssymb}
 \usepackage{enumerate}
 \usepackage{tikz-cd}
 \usepackage{cleveref}
 \usepackage{float}
 \usepackage{graphicx,color}
 \usepackage{fancyvrb}
 \usepackage{subcaption} 
 
\graphicspath{{../}}
 
\theoremstyle{definition}
\newtheorem{Def}{Definition}[section]
\newtheorem{ex}[Def]{Example}

\newtheorem{qu}{Question}

\theoremstyle{plain}

\usepackage{mathrsfs}

\newcommand{\supp}{\operatorname{supp}}
\newcommand{\newt}{\operatorname{Newt}}

\newcommand{\rank}{\operatorname{rank}} 

\newcommand\ratto{\dashrightarrow}

\newcommand{\cM}{{\mathcal M}}

\newcommand{\R}{{\mathbb R}}
\newcommand{\pp}{\mathbb{P}}

\newcommand{\N}{{\mathbb N}}
\newcommand{\Z}{{\mathbb Z}}

\newcommand{\sph}{{\mathbb S}}

\title{A generalization of the space of complete quadrics}
\titlemark{Generalization of complete quadrics}



\author{Abeer Al Ahmadieh}
\address{University of Washington, Seattle, WA, USA\\ \email{aka2222@uw.edu}}

\author{Mario Kummer}
\address{Technische Universit\"at Dresden, Germany\\ \email{mario.kummer@tu-dresden.de}}

\author{Miruna-Stefana Sorea}
\address{SISSA, Trieste, Italy and RCMA Lucian Blaga University Sibiu, Romania\\ \email{msorea@sissa.it}}

\begin{document}

\maketitle

\begin{abstract}
  To any homogeneous polynomial $h$ we naturally associate a variety $\Omega_h$ which maps birationally onto the graph $\Gamma_h$ of the gradient map $\nabla h$ and which agrees with the space of complete quadrics when $h$ is the determinant of the generic symmetric matrix. We give a sufficient criterion for $\Omega_h$ being smooth which applies for example when $h$ is an elementary symmetric polynomial. In this case $\Omega_h$ is a smooth toric variety associated to a certain generalized permutohedron. We also give examples when $\Omega_h$ is not smooth. 
\end{abstract}

\section{Introduction and results}
Let $h\in\R[x_1,\ldots,x_n]$ be a homogeneous polynomial of degree $d$. We will always assume that there is no invertible linear change of coordinates $T$ such that $h(Tx)\in\R[x_1,\ldots,x_{k-1}]$. The \emph{gradient map} of $h$ is the rational map $$\nabla h:\pp^{n-1}\ratto\pp^{n-1},\, x\mapsto[\nabla h(x)]=[\frac{\partial}{\partial x_1}h(x):\cdots:\frac{\partial}{\partial x_n}h(x)].$$It is a regular map on the open subset $U\subset\pp^{n-1}$ of all points where $h$ does not vanish. Its graph $\Gamma_h$ is the Zariski closure of all pairs $(x,\nabla h(x))$ in $\pp^{n-1}\times\pp^{n-1}$ with $x\in U$. 
In this note we will study resolutions of singularities of $\Gamma_h$ for certain $h$ and thereby address Question 43 in \cite{questions}:
\setcounter{qu}{42}
\begin{qu}
 Can  we  define  a  generalization  of  the  space  of  complete  quadrics  where  the role  of the  symmetric determinant  is  played by  an  arbitrary hyperbolic  polynomial $h$?  Such a manifold could be a canonical resolution of the graph of the gradient map of $h$.
\end{qu}

One motivation for this question is the study of so-called \emph{hyperbolic exponential families} in \cite{expo} where the gradient map plays a prominent role. While the construction we present in this article is motivated by the theory of hyperbolic polynomials, hyperbolicity plays a subordinate role in the rest of the paper. Instead, other properties such as $M$-convexity will matter.

In the case when $h=\det(X)$ is the determinant of the $n\times n$ generic symmetric matrix $X$, such a resolution of singularities is given by the \emph{space of complete quadrics}. For any integer $0< i< n$ and any symmetric matrix $A\in\sph^n$ we denote by $\wedge^iA\in\sph^{\binom{n}{i}}$ the representing matrix of the linear map $\wedge^i\R^n\to\wedge^i\R^n$ induced by $A$. Note that $\wedge^iA$ is nonzero if $\det(A)\neq0$. Now the space of complete quadrics $\Omega_{\det X}$ is the Zariski closure of all tuples $([A],[\wedge^2A],\ldots,[\wedge^{n-1}A])$ in $\pp(\sph^n)\times\pp(\sph^{\binom{n}{2}})\times\cdots\times\pp(\sph^{\binom{n}{n-2}})\times\pp(\sph^n)$ with $A$ invertible. The projection of $\Omega_{\det X}$ onto the first and the last coordinate is a birational map onto $\Gamma_{\det(X)}$. Moreover, it was shown for example in \cite{laksov} that $\Omega_{\det X}$ is smooth.

In this note we will define a variety $\Omega_h$ for an arbitrary homogeneous polynomial $h\in\R[x_1,\ldots,x_n]$ together with a regular and birational map to $\Gamma_h$ which agrees with the space of complete quadrics when $h=\det(X)$ is the determinant of the generic symmetric matrix. Before we give the definition of $\Omega_h$, we recall the definition of a hyperbolic polynomial.

\begin{Def}
 A homogeneous polynomial $h\in\R[x_1,\ldots,x_n]$ is \emph{hyperbolic} with respect to $e\in\R^n$ if the univariate polynomial $h(te-v)\in\R[t]$ has only real zeros for all $v\in\R^n$. The \emph{hyperbolicity cone} of $h$ at $e$ is $$\Lambda_e(h)=\{v\in\R^n:\, h(te-v)\textrm{ has only nonnegative roots}\}.$$
\end{Def}

The prototype of a hyperbolic polynomial is the determinant of the generic symmetric matrix $\det(X)$. Indeed, since a real symmetric matrix has only real eigenvalues, the polynomial $\det(X)$ is hyperbolic with respect to the identity matrix $I$. The hyperbolicity cone of $\det(X)$ at $I$ is the cone of positive semidefinite matrices.

The entries of $\wedge^{k+1}X$ cut out the variety of symmetric matrices with rank at most $k$. For a real symmetric matrix $A$ the algebraic and geometric multiplicity of an eigenvalue agree. Thus the rank of $A$ equals to the degree of the univariate polynomial $\det(tI+A)$.

In fact the same holds true when we replace $I$ by any positive definite matrix. This shows that we can express the degeneracy locus of the rational map $\pp(\sph^n)\ratto\pp(\sph^{\binom{n}{k+1}}),\,[A]\mapsto[\wedge^{k+1}A]$ in terms of the hyperbolic rank function of $\det(X)$:

\begin{Def}
 Let $h\in\R[x_1,\ldots,x_n]$ be {hyperbolic} with respect to $e\in\R^n$. The \emph{hyperbolic rank function} of $h$ is defined as $$\rank_{h,e}:\R^n\to\N,\, v\mapsto\deg(h(e+tv)).$$
\end{Def}

It was shown in \cite[Lemma~4.4]{Bra11} that $\rank_{h,e}=\rank_{h,a}$ for any $a\in\textrm{int}(\Lambda_e(h))$. 
We let $d=\deg(h)$, $0\leq k<d$ and $v\in\R^n$. Then we have $\rank_{h,e}(v)\leq d-k-1$ if and only if all $k$th order partial derivatives $\frac{\partial^{k}h}{\partial x_{i_1}\cdots\partial x_{i_k}}$ of $h$ vanish in $v$. Lets denote by $D^k_{1},\ldots, D^k_{m_k}$ a basis of the span of all $k$th order partial derivatives of $h$. We consider the rational map \[{\Delta}h:\pp^{n-1}\ratto\pp^{m_1-1}\times\cdots\times\pp^{m_{d-1}-1},\]\[[x]\mapsto([D^1_1(x):\cdots:D^1_{m_1}(x)],\ldots,[D^{d-1}_1(x):\cdots:D^{d-1}_{m_{d-1}}(x)]).\]
We define the variety $\Omega_h$ to be the normalisation of the image of this rational map. 
The projection on the first and the last coordinate gives a birational morphism $\omega_h:\Omega_h\to\Gamma_h$. Moreover, when $h=\det(X)$ is the determinant of the generic symmetric matrix, then $\Omega_{\det(X)}$ is isomorphic to the space of complete quadrics as defined above and thus $\Omega_{\det(X)}$ is smooth in that case. 

Another important example for hyperbolic polynomials are the elementary symmetric polynomials.

\begin{thm}\label{thm:elems}
 Let $\sigma_{d,n}$ be the elementary symmetric polynomial of degree $d$ in $n$ variables. Then $\Omega_{\sigma_{d,n}}$ is a smooth toric variety. 
\end{thm}

It is well-known that $\sigma_{d,n}$ is hyperbolic with respect to every point in the positive orthant. Such polynomials are called \emph{stable}.  The theory of stable polynomials connects nicely to discrete convex analysis \cite{murota}. We denote by $\delta_k\in\Z^n$  the $k$th unit vector.

\begin{Def}\label{def:mconv}
 A nonempty set of integer points $B\subset\Z^n$ is called \emph{$M$-convex} if for all $x,y\in B$ and every index $i$ with $x_i>y_i$, there exists an index $j$ with $x_j<y_j$ such that $x-\delta_i+\delta_j\in B$ and $y+\delta_i-\delta_j\in B$.
\end{Def}

\begin{thm}[Theorem~3.2 in \cite{interl}]
 Let $h\in\R[x_1,\ldots,x_n]$ be a homogeneous stable polynomial. Then the support of $h$ is $M$-convex.
\end{thm}

In Section \ref{sec:results} we will give a sufficient criterion for $\Omega_h$ being smooth when the support of $h$ is $M$-convex. We will apply this criterion for proving Theorem \ref{thm:elems}.

However, there are also stable (and thus hyperbolic) polynomials $h$ for which $\Omega_h$ is not smooth.

\begin{ex}\label{ex:omegasing}
 Consider the polynomial $$h=w (2 x+4 y+7 z) (4 x+2 y+7 z)$$  $$+x^3+11 x^2 y+11 x y^2+y^3+15 x^2 z+46 x y z+15 y^2 z+37 x z^2+37 y z^2+21 z^3.$$One can check that $h$ is stable. Further, using the the computer algebra system {\tt Macaulay2}~\cite{M2}, one checks that $\Omega_h$ is not smooth. 
\end{ex}

\section{A simple polymatroid}
In this section we prepare the proof of Theorem \ref{thm:elems}. Recall that a \emph{polymatroid} on the ground set $[n]=\{1,\ldots,n\}$ is a function $r:2^{[n]}\to\Z_{\geq0}$ such that for all $S,T\subset[n]$ we have:
\begin{enumerate}
 \item $r(S)\leq r(T)$ if $S\subset T$,
 \item $r(S\cup T)+r(S\cap T)\leq r(S)+r(T)$, and
 \item $r(\emptyset)=0$.
\end{enumerate}
If further $r(\{i\})\leq1$ for all $i\in[n]$, then $r$ is called a \emph{matroid}.
The second property is usually called \emph{submodularity}. We call the number $d=r([n])$ the \emph{rank} of $r$. See \cite[Chapter 18]{welsh} for a general reference on the theory of polymatroids. 

\begin{ex}
 Let $h\in\R[x_1,\ldots,x_n]$ be {hyperbolic} with respect to $e\in\R^n$. The function that sends $S\subset[n]$ to $\rank_{h,e}(\sum_{i\in S}\delta_i)$ is a polymatroid \cite[Proposition~3.2]{Bra11}.
\end{ex}

For all $0\leq k\leq d$ the \emph{$k$th truncation} $r_k$ is the polymatroid defined by $$r_{k}(S)=\min(d-k,r(S))$$for all $S\subset[n]$. It follows directly from the definition that the sum of polymatroids is again a polymatroid. We define the following polymatroid $$\overline{r}=r_0+\ldots+r_d.$$
To every polymatroid $r$ one associates the \emph{independence polytope} $$P(r)=\{x\in(\R_{\geq0})^n:\, \sum_{i\in S}x_i\leq r(S)\textrm{ for all }S\subset[n]\}.$$
The goal of this section is to show that for every polymatroid $r$ on $[n]$ the polytope $P(\overline{r})$ is simple. A characterization of polymatroids, whose independence polytope is simple, was given in \cite[Theorem 2]{girlich}. The following lemmata will enable us to apply this criterion.

\begin{Def}
 Let $r$ be a polymatroid on $[n]$. We say that a subset $S\subset[n]$ is \emph{$r$-inseparable} if for every two disjoint and nonempty subsets $S_1,S_2\subset[n]$ with $S=S_1\cup S_2$ we have $r(S)<r(S_1)+r(S_2)$.
\end{Def}

\begin{rem}
 If $|S|\leq1$, then $S$ is $r$-inseparable for every polymatroid $r$.
\end{rem}

\begin{lemma}\label{lem:sepsum}
 Let $r,r'$ be polymatroids on $[n]$. If $S\subset[n]$ is $r$-inseparable, then $S$ is $(r+r')$-inseparable.
\end{lemma}

\begin{proof}
 Assume that $S$ is not $(r+r')$-inseparable.
 Let $\emptyset\neq S_1,S_2\subset[n]$ such that $S$ is the disjoint union of $S_1$ and $S_2$. If $r(S)+r'(S)\geq r(S_1)+r'(S_1)+r(S_2)+r'(S_2)$, then by submodularity of $r'$ we get $r(S)\geq r(S_1)+r(S_2)$ which shows that $S$ is not $r$-inseparable.
\end{proof}

\begin{rem}\label{rem:loop}
 Let $|S|\geq2$ and let $x\in[n]$ be a \emph{loop} of $r$, i.e. $r(\{x\})=0$. If $x\in S$, then $S$ is not $r$-inseparable: $r(S)=r(S\setminus\{x\})+r(\{x\}).$
\end{rem}

\begin{lemma}\label{lem:loop}
 Let $S\subset[n]$ with $|S|\geq2$ and $r$ a polymatroid on $[n]$. Then $S$ is $\overline{r}$-inseparable if and only if $S$ does not contain a loop of $r$.
\end{lemma}

\begin{proof}
 We first observe that $x\in[n]$ is a loop of $r$ if and only if $x$ is a loop of all truncations $r_k$ and thus of $\overline{r}$. Now the ``only if'' direction follows from Remark \ref{rem:loop}. For the ``if'' direction assume that $S$ does not contain any loop of $r$. By Lemma \ref{lem:sepsum} it suffices to show that $S$ is $r_{d-1}$-inseparable. This is clear since $$r_{d-1}(S)=1<2=r_{d-1}(S_1)+r_{d-1}(S_2)$$for all nonempty subsets $S_1,S_2\subset S$.
\end{proof}

\begin{lemma}\label{lem:simplecondition1}
 Let $r$ be a polymatroid on $[n]$ of rank $d$. Let $S,T\subset[n]$ such that \begin{enumerate}
 \item $S\cap T\neq\emptyset$, $S\not\subset T$, $T\not\subset S$,
 \item $\overline{r}(S\cap T)<\overline{r}(S)$, $\overline{r}(S\cap T)<\overline{r}(T)$, and
 \item the sets $S,T,S\cup T$ are $\overline{r}$-inseparable.
 \end{enumerate}
 Then $\overline{r}(S\cap T)+\overline{r}(S\cup T)<\overline{r}(S)+\overline{r}(T)$.
\end{lemma}

\begin{proof}
 We proceed by induction on $d$. We first show that for $d\leq1$ there are no subsets $S,T\subset[n]$ satisfying $(1),(2),(3)$. If $d=0$, then $r$ and $\overline{r}$ are both the zero function. Thus there are no subsets $S,T\subset[n]$ satisfying $(2)$. If $d=1$, we still have $r=\overline{r}$. Condition $(1)$ implies that $|S|\geq2$. Thus $(3)$ and Lemma \ref{lem:loop} imply that $S$ contains no loop of $r$. Therefore, we have $r(S)=r(S\cap T)=1$ contradicting $(2)$.
 
 Now let $d>1$ and assume that the claim is true for the polymatroid $r_1$ of rank $d-1$. We assume for the sake of a contradiction that $S,T\subset[n]$ satisfy $(1),(2),(3)$ but $\overline{r}(S\cap T)+\overline{r}(S\cup T)=\overline{r}(S)+\overline{r}(T)$. Again $(1)$ implies that $|S|\geq2$. So by $(3)$ and Lemma \ref{lem:loop} the set $S\cup T$ contains no loop of $r$. Since $d>1$, this implies that $S\cup T$ contains no loop of $r_1$ as well. Thus again by Lemma \ref{lem:loop} the sets $S,T,S\cup T$ are $\overline{r_1}$-inseparable. By submodularity and because $\overline{r}=r+\overline{r_1}$ we have $$r(S)+r(T)=r(S\cap T)+r(S\cup T)\textrm{ and }\overline{r_1}(S)+\overline{r_1}(T)=\overline{r_1}(S\cap T)+\overline{r_1}(S\cup T).$$So by induction hypothesis we have without loss of generality that $\overline{r_1}(S\cap T)=\overline{r_1}(S)$, which implies $r_1(S\cap T)=r_1(S)$, and $r(S\cap T)<r(S)$. Thus we must have $r(S)=d$ and the equation $$d+r(T)=r(S)+r(T)=r(S\cap T)+r(S\cup T)=r(S\cap T)+d$$ implies that $r(T)=r(S\cap T)$. This in turn shows that $\overline{r}(T)=\overline{r}(S\cap T)$ contradicting $(2)$.
\end{proof}

\begin{lemma}\label{lem:simplecondition2}
 Let $r$ be a polymatroid on $[n]$ of rank $d$. Let $k\geq2$ and $S_1,\ldots,S_k\subset[n]$ nonempty and pairwise disjoint. Let $S\subset[n]$ $\overline{r}$-inseparable with $\cup_{i=1}^k S_i\subset S$ and $\overline{r}(\cup_{i=1}^k S_i)= \overline{r}(S)$. Then $\overline{r}(\cup_{i=1}^k S_i)<\sum_{i=1}^k\overline{r}(S_i)$.
\end{lemma}

\begin{proof}
 We first observe that since $|S|\geq2$ and $S$ is $\overline{r}$-inseparable, Lemma \ref{lem:loop} implies that $S$ contains no loop of $r$. Thus each $S_i$ also contains no loop of $r$.

 We proceed again by induction on $d$. If $d=0$, then there every element is a loop contradicting the assumptions. If $d=1$, then  we have \[\overline{r}(\cup_{i=1}^k S_i)=1<2\leq k=\sum_{i=1}^k\overline{r}(S_i).\]
 Now let $d>1$. Then because $S$ contains no loop of $r$, it also contains no loop of $r_1$ which shows that $S$ is $\overline{r_1}$-inseparable. Further $\cup_{i=1}^k S_i\subset S$ and $\overline{r}(\cup_{i=1}^k S_i)= \overline{r}(S)$ imply that $\overline{r_1}(\cup_{i=1}^k S_i)= \overline{r_1}(S)$. By induction hypothesis we have $\overline{r_1}(\cup_{i=1}^k S_i)<\sum_{i=1}^k\overline{r_1}(S_i)$ which implies the claim because $r=r_0$ is submodular.
\end{proof}

\begin{thm}\label{thm:simple}
 Let $r$ be a polymatroid on $[n]$. Then the polytope $P(\overline{r})$ is simple.
\end{thm}

\begin{proof}
 A characterization of simple independence polytopes of polymatroids was given in \cite[Theorem 2]{girlich}. It says that the polytope $P(\overline{r})$ is simple if and only if the conclusion of the two preceding Lemmas \ref{lem:simplecondition1} and \ref{lem:simplecondition2} holds.
\end{proof}

We will be interested in the base polytope of a polymatroid rather than in its independent polytope. If $r:2^{[n]}\to\R$ is a submodular function, then its \emph{base polytope} $B(r)$ is defined as
$$B(r)=\{x\in(\R_{\geq0})^n:\, \sum_{i\in S}x_i\leq r(S)\textrm{ for all }S\subset[n]\textrm{ and }\sum_{i=1}^nx_i=r([n])\}.$$

\begin{cor}\label{cor:simple}
 Let $r$ be a polymatroid on $[n]$. Then the polytope $B(\overline{r})$ is simple.
\end{cor}

\begin{proof}
 Clearly, the base polytope is a face of the independence polytope. Thus the claim follows from Theorem \ref{thm:simple}.
\end{proof}

\begin{rem}\label{cor:minkowski}
 Taking the sum of submodular functions is compatible with taking the Minkowski sum of their base polytopes \cite[Theorem~4.23(1)]{murota}. Thus if $r$ is a polymatroid on $[n]$ of rank $d$, then we have that $B(\overline{r})=B(r_0)+\ldots+B(r_d)$.
\end{rem}

We end this section with describing the polytope $B(\overline{r})$ explicitely when $r$ is the rank function of a matroid. We start with the following easy lemma.

\begin{lemma}
 Let $r=r_\cM$ be the rank function of a matroid $\cM$ of rank $d$ on $[n]$. There is a basis $B$ of $\cM$ such that for all $i\in[n]$ we have $$a_i:=r([i])-r([i-1])=\begin{cases}1&\textrm{ if } i\in B\\0&\textrm{ otherwise.}\end{cases}$$
\end{lemma}

\begin{proof}
 Since $\cM$ is a matroid of rank $d$, we have $a_i\in\{0,1\}$ and $B=\{i\in[n]:\, a_i=1\}$ has cardinality $d$. Let $k_1<\cdots<k_d$ the elements of $B$. We show by induction on $m$ that $I_m:=\{k_1,\ldots,k_m\}$ is independent. Assume that $I_{m-1}$ is independent. Since $r([k_m])=m$, there is an independent subset $I$ of $[k_m]$ of cardinality $m$. Thus there is an element $e\in I\setminus I_{m-1}$ such that $I_{m-1}\cup\{e\}$ is independent. Since $r([k_m-1])=m-1$, we must have $e=k_m$.
\end{proof}

\begin{prop}\label{prop:matroidpolytope}
 Let $r=r_\cM$ be the rank function of a matroid $\cM$ of rank $d$ on $[n]$. The vertices of the polytope $B(\overline{r})$ are exactly those points $v\in\R^n$ whose support is a basis of $\cM$ and whose nonzero entries comprise the numbers $1,\ldots,d$.
\end{prop}

\begin{proof}
 Let $v$ be a vertex of $B(\overline{r})$. Then $v$ is also a vertex of $P(\overline{r})$. Then by \cite[\S18.4, Theorem 1]{welsh} there exists an integer $0\leq k\leq n$ and a bijection $\pi:[n]\to[n]$ such that $v_{\pi(j)}=\overline{r}(\{\pi(1),\ldots,\pi(j)\})-\overline{r}(\{\pi(1),\ldots,\pi(j-1)\})$ if $j\in[k]$ and $v_{\pi(j)}=0$ otherwise. Since $v$ lies in $B(\overline{r})$, we can assume without loss of generality that $k=n$. Further, after relabeling, we can assume that $\pi$ is the identity map. Now let $B=\{k_1,\ldots,k_d\}$ with $k_1<\cdots<k_d$ be the basis of $\cM$ as in the preceding lemma. Then we have $v_j=d-m+1$ if $j=k_m$ and zero if $j\not\in B$. This shows that $v$ is of the desired form.

Conversely, take $v\in\R^n$ whose support $\{i_1,\ldots,i_d\}$ is a basis of $\cM$ such that $v_{i_j}=d-j+1$ for $j=1,\ldots,d$. Since $v_{i_j}=\overline{r}(\{i_1,\ldots,i_j\})-\overline{r}(\{i_1,\ldots,i_{j-1}\})$ for $j=1,\ldots,d$ and all other entries of $v$ are zero, it is a vertex of $P(\overline{r})$ by \cite[\S18.4, Theorem 1]{welsh}. One checks that $\sum_{i=1}^n v_i=\bar{r}([n])$, so $v$ is  a vertex of $B(\overline{r})$.
\end{proof}

\begin{ex}\label{ex:matroidpolytope}
  For instance when $\cM=U(2,4)$ is the uniform matroid on $4$ elements of rank $2$, then $B(r_1)$ is the standard $3$-simplex in $\R^4$ and $B(r_0)$ is the octahedron whose vertices are the permutations of $( 1 ,1, 0, 0)$ (and thus is not simple). The Minkowski sum $B(\overline{r})=B(r_0)+B(r_1)$ is simple by Corollary \ref{cor:simple}. It is the truncated tetrahedron whose vertices are the permutations of $( 2 ,1, 0, 0)$.
\end{ex}

\begin{figure}[h]\centering
\begin{subfigure}[b]{0.30\textwidth}

\begin{tikzpicture}[x  = {(0.9cm,-0.076cm)},
                    y  = {(-0.06cm,0.95cm)},
                    z  = {(-0.44cm,-0.29cm)},
                    scale = 1.8,
                    color = {lightgray}]

  \coordinate (v0_p) at (0, 0, 0);
  \coordinate (v1_p) at (1, 0, 0);
  \coordinate (v2_p) at (0, 1, 0);
  \coordinate (v3_p) at (0, 0, 1);

  \definecolor{edgecolor_p}{rgb}{ 0,0,0 }
  \tikzstyle{facestyle_p} = [fill=none, fill opacity=0.85, preaction={draw=white, line cap=round, line width=1.5 pt}, draw=edgecolor_p, line width=1 pt, line cap=round, line join=round]

  \draw[facestyle_p] (v1_p) -- (v3_p) -- (v0_p) -- (v1_p) -- cycle;
  \draw[facestyle_p] (v0_p) -- (v3_p) -- (v2_p) -- (v0_p) -- cycle;
  \draw[facestyle_p] (v0_p) -- (v2_p) -- (v1_p) -- (v0_p) -- cycle;


  \draw[facestyle_p] (v2_p) -- (v3_p) -- (v1_p) -- (v2_p) -- cycle;



  \definecolor{pointcolor_Latticepointsandverticesofp_0}{rgb}{ 0.2745,0.5882,0.2745 }
  \definecolor{pointcolor_Latticepointsandverticesofp_1}{rgb}{ 0.2745,0.5882,0.2745 }
  \definecolor{pointcolor_Latticepointsandverticesofp_2}{rgb}{ 0.2745,0.5882,0.2745 }
  \definecolor{pointcolor_Latticepointsandverticesofp_3}{rgb}{ 0.2745,0.5882,0.2745 }

  \coordinate (v0_Latticepointsandverticesofp) at (0, 0, 0);
  \coordinate (v1_Latticepointsandverticesofp) at (0, 0, 1);
  \coordinate (v2_Latticepointsandverticesofp) at (0, 1, 0);
  \coordinate (v3_Latticepointsandverticesofp) at (1, 0, 0);

  \fill[pointcolor_Latticepointsandverticesofp_1] (v1_Latticepointsandverticesofp) circle (1 pt);

  \fill[pointcolor_Latticepointsandverticesofp_0] (v0_Latticepointsandverticesofp) circle (1 pt);

  \fill[pointcolor_Latticepointsandverticesofp_2] (v2_Latticepointsandverticesofp) circle (1 pt);

  \fill[pointcolor_Latticepointsandverticesofp_3] (v3_Latticepointsandverticesofp) circle (1 pt);

\end{tikzpicture}
\end{subfigure}
\begin{subfigure}[b]{0.30\textwidth}

\begin{tikzpicture}[x  = {(0.9cm,-0.076cm)},
                    y  = {(-0.06cm,0.95cm)},
                    z  = {(-0.44cm,-0.29cm)},
                    scale = 1.7,
                    color = {lightgray}]

  \coordinate (v0_pr) at (1, 0, 0);
  \coordinate (v1_pr) at (0, 1, 0);
  \coordinate (v2_pr) at (0, 0, 1);
  \coordinate (v3_pr) at (1, 1, 0);
  \coordinate (v4_pr) at (1, 0, 1);
  \coordinate (v5_pr) at (0, 1, 1);

  \definecolor{edgecolor_pr}{rgb}{ 0,0,0 }
  \tikzstyle{facestyle_pr} = [fill=none, fill opacity=0.85, preaction={draw=white, line cap=round, line width=1.5 pt}, draw=edgecolor_pr, line width=1 pt, line cap=round, line join=round]

  \draw[facestyle_pr] (v2_pr) -- (v0_pr) -- (v4_pr) -- (v2_pr) -- cycle;
  \draw[facestyle_pr] (v2_pr) -- (v5_pr) -- (v1_pr) -- (v2_pr) -- cycle;
  \draw[facestyle_pr] (v3_pr) -- (v0_pr) -- (v1_pr) -- (v3_pr) -- cycle;
  \draw[facestyle_pr] (v1_pr) -- (v0_pr) -- (v2_pr) -- (v1_pr) -- cycle;
  \draw[facestyle_pr] (v3_pr) -- (v1_pr) -- (v5_pr) -- (v3_pr) -- cycle;


  \draw[facestyle_pr] (v4_pr) -- (v0_pr) -- (v3_pr) -- (v4_pr) -- cycle;


  \draw[facestyle_pr] (v2_pr) -- (v4_pr) -- (v5_pr) -- (v2_pr) -- cycle;


  \draw[facestyle_pr] (v4_pr) -- (v3_pr) -- (v5_pr) -- (v4_pr) -- cycle;



  \definecolor{pointcolor_Latticepointsandverticesofpr_0}{rgb}{ 0.2745,0.5882,0.2745 }
  \definecolor{pointcolor_Latticepointsandverticesofpr_1}{rgb}{ 0.2745,0.5882,0.2745 }
  \definecolor{pointcolor_Latticepointsandverticesofpr_2}{rgb}{ 0.2745,0.5882,0.2745 }
  \definecolor{pointcolor_Latticepointsandverticesofpr_3}{rgb}{ 0.2745,0.5882,0.2745 }
  \definecolor{pointcolor_Latticepointsandverticesofpr_4}{rgb}{ 0.2745,0.5882,0.2745 }
  \definecolor{pointcolor_Latticepointsandverticesofpr_5}{rgb}{ 0.2745,0.5882,0.2745 }

  \coordinate (v0_Latticepointsandverticesofpr) at (0, 0, 1);
  \coordinate (v1_Latticepointsandverticesofpr) at (0, 1, 0);
  \coordinate (v2_Latticepointsandverticesofpr) at (0, 1, 1);
  \coordinate (v3_Latticepointsandverticesofpr) at (1, 0, 0);
  \coordinate (v4_Latticepointsandverticesofpr) at (1, 0, 1);
  \coordinate (v5_Latticepointsandverticesofpr) at (1, 1, 0);

  \fill[pointcolor_Latticepointsandverticesofpr_0] (v0_Latticepointsandverticesofpr) circle (1 pt);

  \fill[pointcolor_Latticepointsandverticesofpr_2] (v2_Latticepointsandverticesofpr) circle (1 pt);

  \fill[pointcolor_Latticepointsandverticesofpr_4] (v4_Latticepointsandverticesofpr) circle (1 pt);

  \fill[pointcolor_Latticepointsandverticesofpr_1] (v1_Latticepointsandverticesofpr) circle (1 pt);

  \fill[pointcolor_Latticepointsandverticesofpr_3] (v3_Latticepointsandverticesofpr) circle (1 pt);

  \fill[pointcolor_Latticepointsandverticesofpr_5] (v5_Latticepointsandverticesofpr) circle (1 pt);

\end{tikzpicture}
\end{subfigure}
\begin{subfigure}[b]{0.30\textwidth}

\begin{tikzpicture}[x  = {(0.9cm,-0.076cm)},
                    y  = {(-0.06cm,0.95cm)},
                    z  = {(-0.44cm,-0.29cm)},
                    scale = 1,
                    color = {lightgray}]

  \coordinate (v0_p) at (1, 0, 0);
  \coordinate (v1_p) at (0, 1, 0);
  \coordinate (v2_p) at (0, 0, 1);
  \coordinate (v3_p) at (2, 0, 0);
  \coordinate (v4_p) at (2, 1, 0);
  \coordinate (v5_p) at (2, 0, 1);
  \coordinate (v6_p) at (0, 2, 0);
  \coordinate (v7_p) at (1, 2, 0);
  \coordinate (v8_p) at (0, 2, 1);
  \coordinate (v9_p) at (0, 0, 2);
  \coordinate (v10_p) at (1, 0, 2);
  \coordinate (v11_p) at (0, 1, 2);

  \definecolor{edgecolor_p}{rgb}{ 0,0,0 }
  \tikzstyle{facestyle_p} = [fill=none, fill opacity=0.85, preaction={draw=white, line cap=round, line width=1.5 pt}, draw=edgecolor_p, line width=1 pt, line cap=round, line join=round]

  \draw[facestyle_p] (v5_p) -- (v10_p) -- (v9_p) -- (v2_p) -- (v0_p) -- (v3_p) -- (v5_p) -- cycle;
  \draw[facestyle_p] (v11_p) -- (v8_p) -- (v6_p) -- (v1_p) -- (v2_p) -- (v9_p) -- (v11_p) -- cycle;
  \draw[facestyle_p] (v6_p) -- (v7_p) -- (v4_p) -- (v3_p) -- (v0_p) -- (v1_p) -- (v6_p) -- cycle;
  \draw[facestyle_p] (v2_p) -- (v1_p) -- (v0_p) -- (v2_p) -- cycle;


  \draw[facestyle_p] (v8_p) -- (v7_p) -- (v6_p) -- (v8_p) -- cycle;


  \draw[facestyle_p] (v4_p) -- (v5_p) -- (v3_p) -- (v4_p) -- cycle;


  \draw[facestyle_p] (v10_p) -- (v11_p) -- (v9_p) -- (v10_p) -- cycle;


  \draw[facestyle_p] (v4_p) -- (v7_p) -- (v8_p) -- (v11_p) -- (v10_p) -- (v5_p) -- (v4_p) -- cycle;



  \definecolor{pointcolor_Latticepointsandverticesofp_0}{rgb}{ 0.2745,0.5882,0.2745 }
  \definecolor{pointcolor_Latticepointsandverticesofp_1}{rgb}{ 0.2745,0.5882,0.2745 }
  \definecolor{pointcolor_Latticepointsandverticesofp_2}{rgb}{ 0.2745,0.5882,0.2745 }
  \definecolor{pointcolor_Latticepointsandverticesofp_3}{rgb}{ 0.2745,0.5882,0.2745 }
  \definecolor{pointcolor_Latticepointsandverticesofp_4}{rgb}{ 0.2745,0.5882,0.2745 }
  \definecolor{pointcolor_Latticepointsandverticesofp_5}{rgb}{ 0.2745,0.5882,0.2745 }
  \definecolor{pointcolor_Latticepointsandverticesofp_6}{rgb}{ 0.2745,0.5882,0.2745 }
  \definecolor{pointcolor_Latticepointsandverticesofp_7}{rgb}{ 0.2745,0.5882,0.2745 }
  \definecolor{pointcolor_Latticepointsandverticesofp_8}{rgb}{ 0.2745,0.5882,0.2745 }
  \definecolor{pointcolor_Latticepointsandverticesofp_9}{rgb}{ 0.2745,0.5882,0.2745 }
  \definecolor{pointcolor_Latticepointsandverticesofp_10}{rgb}{ 0.2745,0.5882,0.2745 }
  \definecolor{pointcolor_Latticepointsandverticesofp_11}{rgb}{ 0.2745,0.5882,0.2745 }
  \definecolor{pointcolor_Latticepointsandverticesofp_12}{rgb}{ 0.2745,0.5882,0.2745 }
  \definecolor{pointcolor_Latticepointsandverticesofp_13}{rgb}{ 0.2745,0.5882,0.2745 }
  \definecolor{pointcolor_Latticepointsandverticesofp_14}{rgb}{ 0.2745,0.5882,0.2745 }
  \definecolor{pointcolor_Latticepointsandverticesofp_15}{rgb}{ 0.2745,0.5882,0.2745 }

  \coordinate (v0_Latticepointsandverticesofp) at (0, 0, 1);
  \coordinate (v1_Latticepointsandverticesofp) at (0, 0, 2);
  \coordinate (v2_Latticepointsandverticesofp) at (0, 1, 0);
  \coordinate (v3_Latticepointsandverticesofp) at (0, 1, 1);
  \coordinate (v4_Latticepointsandverticesofp) at (0, 1, 2);
  \coordinate (v5_Latticepointsandverticesofp) at (0, 2, 0);
  \coordinate (v6_Latticepointsandverticesofp) at (0, 2, 1);
  \coordinate (v7_Latticepointsandverticesofp) at (1, 0, 0);
  \coordinate (v8_Latticepointsandverticesofp) at (1, 0, 1);
  \coordinate (v9_Latticepointsandverticesofp) at (1, 0, 2);
  \coordinate (v10_Latticepointsandverticesofp) at (1, 1, 0);
  \coordinate (v11_Latticepointsandverticesofp) at (1, 1, 1);
  \coordinate (v12_Latticepointsandverticesofp) at (1, 2, 0);
  \coordinate (v13_Latticepointsandverticesofp) at (2, 0, 0);
  \coordinate (v14_Latticepointsandverticesofp) at (2, 0, 1);
  \coordinate (v15_Latticepointsandverticesofp) at (2, 1, 0);

  \fill[pointcolor_Latticepointsandverticesofp_1] (v1_Latticepointsandverticesofp) circle (1 pt);

  \fill[pointcolor_Latticepointsandverticesofp_4] (v4_Latticepointsandverticesofp) circle (1 pt);

  \fill[pointcolor_Latticepointsandverticesofp_9] (v9_Latticepointsandverticesofp) circle (1 pt);

  \fill[pointcolor_Latticepointsandverticesofp_0] (v0_Latticepointsandverticesofp) circle (1 pt);

  \fill[pointcolor_Latticepointsandverticesofp_3] (v3_Latticepointsandverticesofp) circle (1 pt);

  \fill[pointcolor_Latticepointsandverticesofp_6] (v6_Latticepointsandverticesofp) circle (1 pt);

  \fill[pointcolor_Latticepointsandverticesofp_8] (v8_Latticepointsandverticesofp) circle (1 pt);

  \fill[pointcolor_Latticepointsandverticesofp_11] (v11_Latticepointsandverticesofp) circle (1 pt);

  \fill[pointcolor_Latticepointsandverticesofp_14] (v14_Latticepointsandverticesofp) circle (1 pt);

  \fill[pointcolor_Latticepointsandverticesofp_2] (v2_Latticepointsandverticesofp) circle (1 pt);

  \fill[pointcolor_Latticepointsandverticesofp_5] (v5_Latticepointsandverticesofp) circle (1 pt);

  \fill[pointcolor_Latticepointsandverticesofp_7] (v7_Latticepointsandverticesofp) circle (1 pt);

  \fill[pointcolor_Latticepointsandverticesofp_10] (v10_Latticepointsandverticesofp) circle (1 pt);

  \fill[pointcolor_Latticepointsandverticesofp_12] (v12_Latticepointsandverticesofp) circle (1 pt);

  \fill[pointcolor_Latticepointsandverticesofp_13] (v13_Latticepointsandverticesofp) circle (1 pt);

  \fill[pointcolor_Latticepointsandverticesofp_15] (v15_Latticepointsandverticesofp) circle (1 pt);

\end{tikzpicture}
\end{subfigure}
\caption{Lattice polytopes from Example \ref{ex:matroidpolytope} (left to right): $B(r_0)$, $B(r_1)$ and $B(\overline{r})$. The figures were created using {\tt polymake}~\cite{polymake}.}
\end{figure}
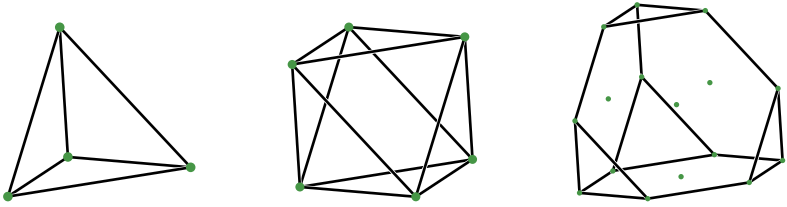

\section{Polynomials with $M$-convex support}
Let $h\in\R[x_1,\ldots,x_n]$ be a homogeneous polynomial of degree $d$ and assume that its support $\supp(h)\subset\Z^n$ is $M$-convex (see Definition \ref{def:mconv}). Recall that the \emph{Newton polytope} $\newt(h)$ of $h$ is defined as the convex hull of $\supp(h)$ in $\R^n$. The statements in the following theorem are standard in the literature of polymatroids. Proofs can be found for example in \cite[\S4.4]{murota}.

\begin{thm}\label{thm:mconvpoly}
 Consider the function $\rho_h:2^{[n]}\to\Z_{\geq0}$ defined by $$\rho_h(S)=\max\{\sum_{i\in S}\alpha_i:\, \alpha\in\supp(h)\}$$for all $S\subset[n]$. Then $\rho_h$ is a polymatroid of rank $d$, $\newt(h)=B(\rho_h)$ and $\supp(h)=B(\rho_h)\cap\Z^n$.
\end{thm}

\begin{rem}
 If $h$ is stable, then all coefficients of $h$ have the same sign, see e.g. \cite[Lemma~4.3]{Bra11}. This implies that for every $e\in(\R_{>0})^n$ we have that
 $$\rho_h(S)=\rank_{h,e}(\sum_{i\in S}\delta_i)$$as there can be no cancellation of terms.
\end{rem}

An intriguing class of polynomials with $M$-convex support are Lorentzian polynomials.

\begin{Def}
 Let $h\in\R[x_1,\ldots,x_n]$ be a homogeneous polynomial of degree $d$ whose support is $M$-convex and all of whose coefficients are nonnegative. Then $h$ is \emph{Lorentzian} if for every $i_1,\ldots,i_{d-2}\in[n]$ the Hessian of the derivative $$\frac{\partial^{d-2}}{\partial x_{i_1}\cdots\partial x_{i_{d-2}}}h$$ has at most one positive eigenvalue.
\end{Def}

\begin{rem}
 Let us clarify the relations between the different classes of polynomials that appeared so far. Let $h\in\R[x_1,\ldots,x_n]$ be a homogeneous polynomial with nonnegative coefficients. Then we have the following implications:
 \begin{eqnarray*}
    &h\textrm{ is stable} \Leftrightarrow h\textrm{ is hyperbolic w.r.t. every }a\in(\R_{>0})^n\\ \Rightarrow& h\textrm{ is Lorentzian}\\
     \Rightarrow&\supp(h) \textrm{ is }M\textrm{-convex}
 \end{eqnarray*}
 There are Lorentzian polynomials that are not stable \cite[Example~2.3]{lorentzian}. Furthermore, not every homogeneous polynomial with $M$-convex support and nonnegative coefficients is Lorentzian.
\end{rem}

\begin{thm}[Theorem~3.10 in \cite{lorentzian}]\label{thm:lorsup}
 A subset $B\subset(\Z_{\geq0})^n$ is $M$-convex if and only if there is a Lorentzian polynomial $h\in\R[x_1,\ldots,x_n]$ with $B=\supp(h)$.
\end{thm}

\begin{lemma}[Corollary~2.11 in \cite{lorentzian}]\label{lem:deri}
 Let $h\in\R[x_1,\ldots,x_n]$ be a Lorentzian polynomial and $e\in(\R_{\geq0})^n$. The derivative $$\textrm{D}_eh=\sum_{i=1}^n e_i \frac{\partial h}{\partial x_i}$$is Lorentzian as well. In particular, the support of $\textrm{D}_eh$ is $M$-convex.
\end{lemma}

\begin{lemma}
 If $h\in\R[x_1,\ldots,x_n]$ is Lorentzian of degree $d$ and $e\in(\R_{>0})^n$, then we have for all $0\leq k\leq d$ that $(\rho_h)_k=\rho_{\textrm{D}^k_eh}.$
\end{lemma}

\begin{proof}
 It suffices to prove the claim in the case $k=1$ because the general case follows from an iterative application of this case.
 
 Let $S\subset[n]$. Since $\textrm{D}_eh$ has degree $d-1$, we have $\rho_{\textrm{D}_eh}(S)\leq d-1$. If $\rho_h(S)=d$, then there is an $\alpha\in\supp(h)$ such that $\sum_{i\in S}\alpha_i=d$. For any $j\in[n]$ with $\alpha_j>0$ we have $\alpha'=\alpha-\delta_j\in\supp(\textrm{D}_eh)$ and thus $\rho_{\textrm{D}_eh}(S)\geq d-1$. Now let $\rho_h(S)<d$ and $\alpha\in\supp(h)$ such that $\sum_{i\in S}\alpha_i=\rho_h(S)$. Since the degree of $h$ is $d$, there must be an index $j\in[n]\setminus S$ such that $\alpha_j>0$. We have $\alpha'=\alpha-\delta_j\in\supp(\textrm{D}_eh)$ and thus $\rho_{\textrm{D}_eh}(S)\geq \rho_h(S)$. If $\beta\in\supp(\textrm{D}_eh)$ satisfies $\rho_{\textrm{D}_eh}(S)=\sum_{i\in S}\beta_i$, then there is a $j\in[n]$ such that $\beta+\delta_j\in\supp(h)$ so $\rho_{\textrm{D}_eh}(S)\leq \rho_h(S)$.
\end{proof}

The following lemma connects the polymatroid $\overline{\rho_h}$ with the variety $\Omega_h$.

\begin{prop}\label{prop:monomials}
 Let $h\in\R[x_1,\ldots,x_n]$ be homogeneous of degree $d$ with $\supp(h)$ being $M$-convex.
 Consider the polymatroid $r=\rho_h$. For each $0\leq k\leq d$ the set $B(r_k)\cap\Z^n$ agrees with the set $B_k$ of all $\alpha\in\Z^n$ such that the monomial $\prod_{i=1}^nx_i^{\alpha_i}$ is in the support of a $k$th order partial derivative of $h$.
\end{prop}

\begin{proof}
 Both $r_k$ and $B_k$ only depend on the support of $h$. Thus we can assume without loss of generality that $h$ is Lorentzian by Theorem \ref{thm:lorsup}. Then for any $e\in(\R_{>0})^n$ we have that $B_k$ is the support of $\textrm{D}^k_eh$ because $h$ has nonnegative coefficients. Thus $B_k$ is $M$-convex by Lemma \ref{lem:deri} and the result follows from Theorem \ref{thm:mconvpoly} and the preceding lemma.
\end{proof}

\section{Preliminaries from algebraic and toric geometry}
In this section we revisit some notions and results from algebraic geometry that will be used in the final section. 
Let $A$ be a nonzero $(m+1)\times(n+1)$ matrix. Recall that a rational map $\pi:\pp^n\ratto\pp^m$ of the form $[x]\mapsto[Ax]$ is called a \emph{linear projection} and that the \emph{centre} of $\pi$ is the linear subspace $E$ of all $[x]\in\pp^n$ such that $Ax=0$. Clearly, the rational map $\pi$ is regular on $\pp^n\setminus E$. Thus if $X\subset\pp^n$ is a projective variety with $X\cap E=\emptyset$, the restriction $f=\pi|_X:X\to\pi(X)$ is a morphism. This map $f$ is \emph{finite}, so in particular it has only finite fibers. See \cite[\S I.5.3]{Shafa77} for the definition and proofs. We will use the following standard facts on finite morphisms which follow for example from \cite[Lemma~14.8]{ha95}.

\begin{lemma}\label{lem:finite}
 Let $f_i:X_i\to Y_i$,  $i=1,2$, be finite morphisms of projective varieties. The product map $f_1\times f_2$ is also finite. If $Y_1=X_2$, then $f_2\circ f_1$ is finite. If $Z\subset X_1$ is closed, then $f_1|_Z:Z\to f_1(Z)$ is finite.
\end{lemma}

Given a projective variety $X$, a \emph{normalisation} of $X$ is a normal variety ${X}^\nu$ with a finite birational morphism $\nu:X^\nu\to X$. The normalisation is unique up to isomorphism. In particular, if $Y\to X$ is a finite birational morphism from a smooth variety $Y$, then $Y$ is the normalisation of $X$ because every smooth variety is normal. See \cite[\S II.5]{Shafa77} for definitions and proofs.

We will especially consider \emph{toric varieties}. The book \cite{toric} gives a comprehensive introduction to toric varieties and we will adopt their notation. For example, given a lattice polytope $P\subset\R^n$, we denote by $X_P$ the associated toric variety \cite[\S2.3]{toric}. Similarly, for any finite set $A\subset\Z^n$ of lattice points we denote by $X_A\subset\pp^{|A|-1}$ the image of the monomial map whose exponents are given by the elements of $A$. Note that in general $X_{P\cap\Z^n}$ is not necessarily isomorphic to $X_P$ but when $P$ is a smooth polytope this is the case by \cite[Proposition~2.4.4]{toric}:
A lattice polytope $P\subset\R^n$ is called \emph{smooth} if its associated toric variety $X_P$ is smooth \cite[\S2.4]{toric}. 
We further have:

\begin{cor}\label{cor:smooth}
 Let $r$ be a polymatroid on $[n]$. Then the polytope $B(\overline{r})$ is a smooth lattice polytope.
\end{cor}

\begin{proof}
 By the Corollary to \cite[\S18.4, Theorem 1]{welsh} the independence polytope $P(\overline{r})$ is a lattice polytope. Since $B(\overline{r})$ is a face of $P(\overline{r})$, it is a lattice polytope as well. Corollary \ref{cor:simple} states that $B(\overline{r})$ is simple and the Submodularity Theorem\footnote{The Submodularity Theorem was proved by Edmonds \cite{edmonds} although not stated in the language of generalized permutahedra. An alternative, combinatorial proof is given in \cite[Theorem~1.2]{submod}.} \cite{edmonds} states that $B(\overline{r})$ is a so-called \emph{generalized permutohedron}. Now the claim follows from  \cite[Corollary~3.10]{post} which says that a simple lattice polytope, which is a generalized permutohedron, is smooth.
\end{proof}

\begin{rem}\label{rem:nonsmooth}
 Let $h\in\R[x_1,\ldots,x_n]$ be homogeneous of degree $d$ with $\supp(h)$ being $M$-convex and let $B_k$ the set of all $\alpha\in\Z^n$ such that the monomial $\prod_{i=1}^nx_i^{\alpha_i}$ is in the support of a $k$th order partial derivative of $h$. Then it follows from Proposition \ref{prop:monomials} and Corollary \ref{cor:smooth} that the Minkowski sum $B_1+\ldots+B_{d-1}$ is the set of lattice points in a smooth polytope.
 In general, if we drop the assumption of $M$-convexity, this is no longer true. Consider for example $h=a\cdot x_1x_2^2+b\cdot x_3^3$ with nonzero $a,b$. Then $B_1+B_2$ is the set of lattice points in a simple polytope that is not smooth. 
\end{rem}

\section{A sufficient criterion for smoothness}\label{sec:results}

Let $h\in\R[x_1,\ldots,x_n]$ be a homogeneous polynomial of degree $d$ whose support is $M$-convex with nonnegative coefficients and $r=\rho_h$. Recall that we denote by $D^k_{1},\ldots, D^k_{m_k}$ a basis of the span of all $k$th order partial derivatives of $h$. For all $1\leq k<d$ consider the rational map \[{\Delta^k}h:\pp^{n-1}\ratto\pp^{m_k-1},\,[x]\mapsto[D^k_1(x):\cdots:D^k_{m_k}(x)].\]

By Proposition \ref{prop:monomials} we can decompose the map $\Delta^k h$ as $\pi_k\circ f_k$ where $f_k$ is the monomial map associated to the polytope $B(r_k)$ (whose image is $X_{B(r_k)\cap\Z^n}$) and $\pi_k$ the linear projection given by summing the monomials in each $D^k_i$.

\begin{ex}
 Let $h=x_1^2x_2+x_1x_2^2+x_1^2x_3+x_1x_2x_3+x_2^2x_3$. Then $\Delta^1h(x)$ equals:
 $$[2x_1x_2+x_2^2+2x_1x_3+x_2x_3:x_1^2+2x_1x_2+x_1x_3+2x_2x_3:x_1^2+x_1x_2+x_2^2].$$ We further have $f_1:\pp^2\ratto\pp^4,[x_1:x_2:x_3]\mapsto[x_1^2:x_1x_2:x_1x_3:x_2^2:x_2x_3]$. We label the coordinates on $\pp^4$ by $z_{ij}$ where $i$ and $j$ keep track of the exponent of $x_1$ and $x_2$ respectively. The image $X$ of $f_1$ is cut out by $z_{10}z_{02}-z_{11}z_{01},z_{11}z_{10}-z_{20}z_{01}$ and $z_{11}^2-z_{20}z_{02}$. Furthermore, the projection $\pi_1$ sends $[z_{20}:z_{11}:z_{10}:z_{02}:z_{01}]$ to $$[2z_{11}+z_{02}+2z_{10}+z_{01}:z_{20}+2z_{11}+z_{10}+2z_{01}:z_{20}+z_{11}+z_{02}].$$The centre of $\pi_1$ is spanned by $[0:-1:0:1:1]$ and $[1:-1:1:0:0]$ and is disjoint from $X$. Thus $\pi_1$ restricts to a finite morphism $X\to\pp^2$.
\end{ex}

\begin{prop}\label{prop:matroid}
 If the centre of the linear projection $\pi_k$ is disjoint from $X_{B(r_k)\cap\Z^n}$ for each $1\leq k<d$, then $\Omega_h$ is smooth. More precisely, it is isomorphic to the smooth toric variety $X_{B(\overline{r_1})}$.
\end{prop}

\begin{proof}
 Let $P=B(\overline{r_1})$. By definition $\Omega_h$ is the normalisation of the image $Y\subset\prod_{i=1}^{d-1}\pp^{m_i-1}$ of the birational map $\Delta h(x)=(\Delta^1h(x),\ldots,\Delta^{d-1}h(x))$. Consider the rational map $f:\pp^{n-1}\ratto\prod_{i=1}^{d-1}\pp^{|B(r_i)\cap\Z^n|-1}$ given by $f(x)=(f_1(x),\ldots,f_{d-1}(x))$ and let $\pi=\prod_{i=1}^{d-1}\pi_i$. By construction $\Delta h$ factors as $\pi\circ f$. If we compose $f$ with the Segre embedding of $\prod_{i=1}^{d-1}\pp^{|B(r_i)\cap\Z^n|-1}$, we obtain the monomial map associated to the Minkowski sum $\sum_{i=1}^{d-1}B(r_i)$ which is $P$ by Remark \ref{cor:minkowski}. Thus the image of $f$ can be identified with $X_{P\cap\Z^n}$ and $Y$ is the image of $X_{P\cap\Z^n}$ under the rational map $\pi:\prod_{i=1}^{d-1}\pp^{|B(r_i)\cap\Z^n|-1}\ratto\prod_{i=1}^{d-1}\pp^{m_i-1}$. Because $P$ is smooth by Corollary \ref{cor:smooth} the variety $X_{P\cap\Z^n}$ is the smooth toric variety $X_P$. Letting $E_i$ be the centre of $\pi_i$, this rational map $\pi$ is regular on $U=\prod_{i=1}^{d-1}(\pp^{|B(r_i)\cap\Z^n|-1}\setminus E_i)$. Since the projection of $X_{P}$ on the $i$th factor $\pp^{|B(r_i)\cap\Z^n|-1}$ is $X_{B(r_i)\cap\Z^n}$, which is disjoint from $E_i$ by assumption, it follows that $X_{P}\subset U$. Thus restricting $\pi$ gives a surjective morphism $p=\pi|_{X_P}:X_{P}\to Y$ which is finite since each $\pi_i|_{X_{B(r_i)\cap\Z^n}}$ is finite and by Lemma \ref{lem:finite}. Since $\Delta h=p\circ f$ is birational and $f$ is birational, it follows that $p$ is also birational. Thus $X_P$ is the normalisation of $Y$.
\end{proof}

\begin{ex}
 Consider the polynomial $$h=x^3+11 x^2 y+11 x y^2+y^3+15 x^2 z+46 x y z+15 y^2 z+37 x z^2+37 y z^2+21 z^3$$ $$+w(29x^2+90xy+29y^2+150xz+150yz+137z^2).$$ One can check that $h$ is stable. Note that $h$ has the same support as the polynomial in Example \ref{ex:omegasing} but different coefficients. Using {\tt Macaulay2}~\cite{M2}, one checks that the conditions of Proposition \ref{prop:matroid} are fulfilled and thus $\Omega_h$ is smooth. It is the toric variety associated to the triangular frustum whose vertices are obtained by permuting the last three entries of $(0,3,0,0)$ and $(2,1,0,0)$. Here the coordinates correspond to the variables in alphabetic order.
\end{ex}

Now let $h=\sigma_{d,n}$ be the elementary symmetric polynomial of degree $d$. 

\begin{lemma}\label{lem:elem}
 The centre of the linear projection $\pi_k$ is disjoint from $X_{B(r_k)\cap\Z^n}$ for each $1\leq k<d$.
\end{lemma}

\begin{proof}
 The lattice points of $B(r_k)$ are exactly the points $v\in\R^n$ with $d-k$ entries equal to $1$ and all other entries $0$. Thus $B(r_k)$ is the hypersimplex $\Delta_{d-k}$. We denote $X=X_{\Delta_{d-k}\cap\Z^n}\subset\pp^{\binom{n}{d-k}-1}$ and we label the coordinates on $\pp^{\binom{n}{d-k}-1}$ by $z_S$ for $S\subset[n]$ of size $d-k$. Every $k$th order derivative of $\sigma_{d,n}$ is an elementary symmetric polynomial of degree $d-k$ in the variables indexed by some subset $T\subset[n]$ of size $n-k$. Thus the centre $E$ of $\pi_k$ is the common zero set of all linear forms $L_T=\sum_{S\subset T,\, |S|=d-k} z_S$ for subsets $T\subset[n]$ of size $n-k$. The statement of \cite[Lemma~6.4]{expo} is that $X$ is disjoint from the common zero set $E'$ of all linear forms $H_i=\sum_{S\subset [n]\setminus{i},\, |S|=d-k} z_S$ for $i\in[n]$. We have for all $i\in[n]$: $$\binom{n+k-1-d}{n-d}\cdot H_i=\sum_{T\subset[n]\setminus\{i\},\, |T|=n-k}L_T.$$ This implies $E\subset E'$ and thus $X$ is also disjoint from $E$.
\end{proof}

\begin{proof}[Proof of Theorem \ref{thm:elems}]
 This follows from Lemma \ref{lem:elem} and Proposition \ref{prop:matroid}.
\end{proof}

\begin{rem}
 By Proposition \ref{prop:matroidpolytope}, we have that $\Omega_{\sigma_{d,n}}$ is the smooth toric variety $X_P$ where $P$ is the convex hull of all permutations of $(1,\ldots,d-1,0,\ldots,0)\in\R^n$.
\end{rem}

\begin{rem}
 Fix some $M$-convex set $S\subset(\Z_{\geq0})^n$ and let $V$ be the vector space of polynomials $h$ with $\supp(h)\subseteq S$. Note that there is an integer $d$ such that any $h\in V$ is homogeneous of degree $d$. There are matrices $A^k_h$ whose entries depend linearly on the coefficients of $h$ such that, when $\supp(h)= S$, the centre of $\pi_k$ is the set of all $[x]$ with $A^k_hx=0$. Consider the incidence correspondence $$\Sigma_k=\{([h],[x])\in\pp(V)\times X_{B(r_k)\cap\Z^n}:\, A^k_hx=0\}.$$This is a projective variety. Thus the projection of $\Sigma_k$ onto the first factor is a closed subvariety $Y_k$ of $\pp(V)$. By construction the criterion from Proposition \ref{prop:matroid} applies to a polynomial $h\in V$ with $\supp(h)=S$ if and only if $[h]$ is not contained in any of the $Y_k$. Therefore, depending on $S$, either Proposition \ref{prop:matroid} applies to no polynomial with support $S$, or to a generic such polynomial. We say that $S$ is \emph{torically smoothable} if the latter is the case. The support of the elementary symmetric polynomial $\sigma_{d,n}$ is torically smoothable by Lemma \ref{lem:elem}. Based on  experiments we conjecture that $S$ is torically smoothable at least when $S$ contains the support of $\sigma_{d,n}$. This condition is empty when $d>n$.
\end{rem}

\begin{rem}
 If $h$ has nonnegative coefficients, then we can assume the same for each $D^k_i$. Then the linear projection $\pi_k$ is at least regular on the nonnegative part of $X_{B(r_k)}$ as
 there can be no cancellation of terms. Thus we have at least a regular map on the nonnegative part of $X_{B(\overline{r_1})}$ that maps birationally onto the graph $\Gamma_{h,+}$ of $\nabla h$ restricted to the nonnegative orthant. 
 In general, even when $h$ is stable, we cannot expect $\pi_k$ to be regular on all of $X_{B(r_k)}$. Take for instance the stable polynomial from Example \ref{ex:omegasing}. In this case $B(r_2)$ is the triangular frustum whose vertices are obtained by permuting the last three entries of $(0,2,0,0)$ and $(1,1,0,0)$. Using  {\tt Macaulay2}~\cite{M2} one checks that the centre of $\pi_2$ intersects the toric variety $X_{B(r_2)}$ in a real point of the torus orbit corresponding to the face with vertices $(1,1,0,0),(1,0,1,0)$ and $(1,0,0,1)$.
\end{rem}

\begin{rem}
 Let $h\in\R[x_1,\ldots,x_n]$ be hyperbolic with respect to $e \in\R^n$. 
 In the spirit of the preceding remark one can speculate whether hyperbolicity of $h$ guarantees smoothness of $\Omega_h$ at least at some distinguished subset. To make this more precise let $U\subset\pp^{n-1}$ be the set of all $[p]$ such that $p$ is in the interior of $\Lambda_e(h)$. Since $h$ does not vanish on $U$, the gradient map $\nabla h$ is regular on $U$. We can thus consider the subset $C=\omega_h^{-1}(\nabla h(U))$ of $\Omega_h$. We think it is reasonable to ask whether the Euclidean closure of $C$ contains only smooth points of $\Omega_h$. 
\end{rem}

\section*{Acknowledgements} 
We thank O.~Marigliano, M.~Micha\l{}ek, K.~Ranestad, T.~Seynnaeve, and B.~Sturmfels for coordinating the collaboration project ``Linear Spaces of Symmetric Matrices'' and inspiring the present work. Special thanks go to anonymous referees, whose valuable suggestions significantly improved the quality of the paper.

\end{document}